\documentclass[12pt,a4paper]{amsart}
\usepackage{amscd}
\theoremstyle{plain} 

\begingroup

\newtheorem{theorem}{Theorem}[section]

\newtheorem{cor}[theorem]{Corollary}

\newtheorem{lemma}[theorem]{Lemma}

\endgroup

\theoremstyle{definition}

\newtheorem{definition}[theorem]{Definition}

\theoremstyle{remark}

\begin{document}

\title[Subanalytic triangulation]
{A subanalytic triangulation theorem for real analytic orbifolds}

\author{Marja Kankaanrinta}
\address{Department of Mathematics \\
         P.O. Box 400137\\ University of Virginia\\
         Charlottesville VA 22904 - 4137\\ U.S.A.}
\email{mk5aq@@virginia.edu}

\date{\today}

\subjclass[2010]{57R05, 57R18}

\keywords{subanalytic, triangulation, orbifold}

\begin{abstract}  Let $X$ be a real analytic orbifold. Then each
stratum of $X$ is a subanalytic subset of $X$. We show that
$X$ has a unique subanalytic triangulation compatible with the
strata of $X$. We also show that every ${\rm C}^r$-orbifold,
$1\leq r\leq \infty$, has a real analytic structure. 
This allows us to triangulate differentiable orbifolds.
The results generalize the subanalytic triangulation theorems
previously known for quotient orbifolds.
\end{abstract}

\maketitle
\section{Introduction}
\label{intro}
\noindent    This paper concerns with subanalytic triangulations of real analytic orbifolds.
To every point $x$ in a real analytic orbifold  $X$ of dimension $n$ one can
associate a local group $G_x$, which is unique up to isomorphism.
The sets $X_{(H)}=\{ x\in X\mid G_x\cong H\}$, where $H$ is a finite group,
form a subanalytic stratification for $X$. 
The main result, Theorem \ref{thetheorem},  says that
every real analytic orbifold $X$ has a unique subanalytic triangulation
compatible with the strata of $X$.

Orbit spaces of real analytic manifolds by real analytic proper almost free
actions of Lie groups are called {\it quotient orbifolds}. 
Subanalytic 
triangulations of orbit spaces of proper real analytic  $G$-manifolds 
where $G$ is a Lie group are already known, see \cite{MS1} and Theorem 7.7
in \cite {Ka}, as well as Theorem 3.3 in \cite{MS2} for the case of a compact Lie group.
These results cover the case of quotient orbifolds.  The result of this paper
generalizes the previous results by providing a subanalytic triangulation for all 
real analytic orbifolds.

We also show that every ${\rm C}^r$-orbifold, $1\leq r\leq \infty$,  can be
given a real analytic structure, by first giving a ${\rm C}^\infty$ structure to
all ${\rm C}^r$-orbifolds, $1\leq r<\infty$, and then equipping every
${\rm C}^\infty$-orbifold with a real analytic structure
(Theorems \ref{comba1} and \ref{combanal}). By using the
subanalytic triangulation result for real analytic orbifolds, we obtain a
"subanalytic" triangulation compatible with the strata  for
any ${\rm C}^r$-orbifold $X$, $1\leq r\leq\infty$. "Subanalytic"
triangulations are known to exist for orbit spaces, see \cite{MS2}.

\section{Preliminaries}
\label{preli}
\noindent   

\noindent We begin by recalling the definition and some basic properties
of an orbifold.

\begin{definition}
\label{ensin}
Let $X$ be a topological space and let $n> 0$.
{\begin{enumerate}
\item An $n$-dimensional {\it orbifold chart} for an open subset $V$
of  $X$ is a triple
$(\tilde{V}, G,\varphi)$ such that
{\begin{enumerate}
\item $\tilde{V}$ is a connected open subset of ${\mathbb{R}}^n$,

\item $G$ is a finite group of homeomorphisms acting  on $\tilde{V}$, let
${\rm ker}(G)$ denote the subgroup of $G$ acting trivially on $\tilde{V}$.

\item $\varphi\colon \tilde{V}\to V$ is a $G$-invariant map inducing a
homeomorphism from $\tilde{V}/G$ onto  $V$.
\end{enumerate}}

\item  If $V_i\subset V_j$, an {\it embedding} $(\lambda_{ij}, h_{ij})
\colon (\tilde{V}_i, G_i, \varphi_i)\to
(\tilde{V}_j, G_j,\varphi_j)$ between two orbifold charts is 
{\begin{enumerate}
\item an injective homomorphism $h_{ij}\colon G_i\to G_j$, 
such that $h_{ij}$ is an isomorphism from ${\rm ker}(G_i)$ to
${\rm ker}(G_j)$,
and
\item an equivariant embedding
$\lambda_{ij}\colon \tilde{V}_i\to \tilde{V}_j$ such that $\varphi_j\circ
\lambda_{ij}=\varphi_i$. (The equivariantness means that 
$\lambda_{ij}(gx)=h_{ij}(g)\lambda_{ij}(x)$ for every $g\in G_i$
and every $x\in \tilde{V}_i$.)
\end{enumerate}}

\item  An {\it orbifold atlas} on $X$ is a family ${\mathcal V}=\{ (\tilde{V_i}, G_i, \varphi_i)\}_{
i\in I}$ of orbifold charts such that
{\begin{enumerate}
\item $\{ V_i\}_{i\in I}$ is a covering of $X$,

\item  given two charts $(\tilde{V}_i, G_i, \varphi_i)$ and
$(\tilde{V}_j, G_j, \varphi_j)$ and a point $x\in V_i\cap V_j$, there exists
an open neighborhood $V_k\subset V_i\cap V_j$ of $x$ and a chart
$(\tilde{V}_k, G_k, \varphi_k)$ such that there are embeddings 
$(\lambda_{ki}, h_{ki})\colon (\tilde{V}_k,G_k, \varphi_k)\to
(\tilde{V}_i, G_i,\varphi_i)$ and
$(\lambda_{kj}, h_{kj})\colon (\tilde{V}_k,G_k, \varphi_k)\to
(\tilde{V}_j, G_j,\varphi_j)$.
\end{enumerate}}

\item An atlas ${\mathcal U}$ is called a {\it refinement} of another atlas
${\mathcal W}$ if for every chart in ${\mathcal U}$ there exists an embedding
into some chart of ${\mathcal W}$. Two orbifold atlases 
having a common refinement are called
{\it equivalent}. 
\end{enumerate}}
\end{definition}
 
 \begin{definition}
 An $n$-dimensional {\it orbifold} is a paracompact Hausdorff space $X$
 equipped with an equivalence class of $n$-dimensional orbifold atlases.
 \end{definition}

 The sets $V\in {\mathcal V}$ are called {\it basic open sets} in $X$.
 
 An orbifold $X$ is called {\it reduced}, if  for every orbifold chart $(
 \tilde{V}, G, \varphi)$ of $X$, the group $G$ acts effectively on 
 $\tilde{V}$.

 An orbifold is  called a {\it  ${\rm C}^r$-orbifold},  $1\leq r\leq\omega$
 (where ${\rm C}^\infty$ means smooth and ${\rm C}^\omega$ means real
 analytic), if each $G_i$ acts via ${\rm C}^r$-diffeomorphisms on 
 $\tilde{V}_i$ and if each embedding $\lambda_{ij}\colon {\tilde{V}}_i\to
 \tilde{V}_j$ is differentiable of degree $r$.

An $n$-dimensional orbifold $X$  is called {\it locally smooth}, if 
 for each $x\in X$ there is an orbifold chart $(\tilde{U}, G,\varphi)$
 with  $x\in U=\varphi(\tilde{U})$ and such that
 the action of $G$ on $\tilde{U}\cong {\mathbb{R}}^n$  
 is orthogonal.  By the differentiable
 slice theorem (see Proposition 2.2.2 in \cite{Pa2} and, for the real 
 analytic case, Theorem 2.5 in  \cite{Ka}),
 every ${\rm C}^r$-orbifold, $1\leq r\leq \omega$, is locally smooth.

We assume that every  orbifold  has only countably many
connected components. It follows that our orbifolds are second countable.
Moreover, all orbifolds are paracompact, and
for any orbifold, the dimension equals the covering dimension
(i.e., the topological dimension).

In Section \ref{triangle} we will make use of the following result, originally due to J. Milnor.

\begin{theorem}
\label{palais}
Let $X$ be a paracompact space with covering dimension $n$ and let
$\{ U_\alpha\}$ be an open cover of $X$. Then there is an open cover
$\{ O_{i\beta}\mid \beta\in B_i, i=0,\ldots, n\}$ of $X$ refining $\{ U_\alpha\}$
such that $O_{i\beta}\cap O_{i\beta'}=\emptyset$ if $\beta\not=\beta'$.
\end{theorem}

\begin{proof}
\cite{Pa}, Theorem 1.8.2.
\end{proof}

In the proof of the previous theorem each set $B_i$ is the set of unordered
$(i+1)$-tuples from the indexing set of the $ U_\alpha$. Thus, if the indexing
set of the $U_\alpha$ is countable, it follows that we can assume each $B_i$
to be countable.

A map $f\colon X\to Y$ between two ${\rm C}^r$-orbifolds is called a ${\rm C}^r$
{\it orbifold map}  if for every $x\in X$ there are orbifold charts $(\tilde{U}, G,\varphi)$
and $(\tilde{V}, H,\psi)$, where $x\in U$ and $f(x)\in V$,  a homomorphism
$\theta\colon G\to H$ and a $\theta$-equivariant map $\tilde{f}\colon
\tilde{U}\to\tilde{V}$ such that the following diagram commutes:

$$\begin{CD}
\tilde{U}
@>\tilde{f}>>  \tilde{V}\\
@VVV     @VVV\\
\tilde{U}/G
@>{}>>  \tilde{V}/H\\
@VVV     @VVV\\
 U      
 @>f\vert U>>   V\\
\end{CD}.
$$

A ${\rm C}^r$-map $f\colon X\to Y$ is called a {\it ${\rm C}^r$-diffeomorphism},
if it is a bijection and if the inverse map $f^{-1}$ is a ${\rm C}^r$-map.

\section{Orbifold stratification}
\label{orbstrat}

\noindent 
Let $X$ be a ${\rm C}^r$-orbifold, $1\leq r\leq\omega$.
Let $x\in X$ and let $(\tilde{V}, G, \varphi)$ and $(\tilde{U}, H,
\psi)$ be orbifold charts such that $x\in \varphi(\tilde{V})\cap
\psi(\tilde{U})$. Let $\tilde{x}\in \tilde{V}$ and $\tilde{y}\in\tilde{U}$
be such that $x=\varphi(\tilde{x})=\psi(\tilde{y})$. Let $G_{\tilde{x}}$ and
$H_{\tilde{y}}$ be the isotropy subgroups at $\tilde{x}$ and
$\tilde{y}$, respectively. Then $G_{\tilde{x}}$ 
and $H_{\tilde{y}}$ are isomorphic. 
It follows that it is possible to associate to
every point $x\in X$ a finite group $G_x$,
well-defined up to an isomorphism of groups, and  called the {\it local group} of $x$.

For a finite group $H$, we let
$$
X_{(H)}=\{ x\in X\mid G_x \cong H\}.
$$
The sets $X_{(H)}$ are called the {\it strata} of $X$. 

We point out, that if $X$ is an $n$-dimensional reduced orbifold, 
then $X$ can be considered as the orbit space of a manifold with the
action of the orthogonal group ${\rm O}(n)$. Thus, in this case, the local
groups are defined up to conjugacy by an element in ${\rm O}(n)$.

\section{Subanalytic subsets of real analytic orbifolds}

\noindent For subanalytic subsets of real analytic manifolds, see
for example \cite{BM} and \cite{Hi1}. Subanalytic subsets of real 
analytic orbifolds were introduced in \cite{Ka2}. We recall the
definitions.

\begin{definition}
\label{subset}
Let $X$ be a real analytic orbifold. A subset $A$ of $X$ is called
{\it subanalytic} if  for every  point $x$ of $X$ there is an 
orbifold chart $(\tilde{V},G,\varphi)$ of $X$ such that
$x\in V=\varphi(\tilde{V})$ and $\varphi^{-1}(A\cap V)$ is a 
subanalytic subset of $\tilde{V}$. 
\end{definition}

Subanalytic orbifold maps are defined in the same way as
${\rm C}^r$-maps, $1\leq r\leq\omega$:

\begin{definition}
\label{submap}
A  map $f\colon X\to Y$ between two real analytic orbifolds is 
called  
 {\it subanalytic}
if for every $x\in X$ there are orbifold charts $(\tilde{U}, G,\varphi)$
and $(\tilde{V}, H,\psi)$, where $x\in U$ and $f(x)\in V$,  a homomorphism
$\theta\colon G\to H$ and a $\theta$-equivariant 
subanalytic map $\tilde{f}\colon
\tilde{U}\to\tilde{V}$ making the following diagram commute:

$$\begin{CD}
\tilde{U}
@>\tilde{f}>>  \tilde{V}\\
@VVV     @VVV\\
\tilde{U}/G
@>{}>>  \tilde{V}/H\\
@VVV     @VVV\\
 U      
 @>f\vert U>>   V\\
\end{CD}.
$$
\end{definition}

\begin{lemma}
\label{substrata}
Let $X$ be a real analytic orbifold. Then all the strata $X_{(H)}$ are
subanalytic subsets of $X$.
\end{lemma}

\begin{proof} This follows from the fact that if $G$ is a finite group acting
real analytically on a real analytic manifold $M$, then the sets
$\{ x\in M\mid G_x=gHg^{-1},\,\, {\rm for\,\, some\,\, } g\in G\}$
are subanalytic subsets of $M$, see e.g. Lemma 3.2 in \cite{MS2}.
Then $\{ x\in M\mid G_x\cong H\}$ is subanalytic as a finite union
of such subanalytic sets.
\end{proof}

The following lemmas follow immediately from the corresponding
results for subanalytic maps from a real analytic manifold to a
euclidean space, see e.g. Lemmas 4.25 and 4.28 in \cite{Ka}.

\begin{lemma}
\label{orbtulo}
Let $X$ be a real analytic orbifold  and let $h\colon X\to {\mathbb{R}}$
and $f\colon X\to {\mathbb{R}}^n$, where $n\in {\mathbb{N}}$, be subanalytic
maps. Then the product $hf\colon X\to {\mathbb{R}}^n$ is a subanalytic map.
If $h(x)\not= 0$ for every $x\in X$, then the quotient $f/h\colon X\to {\mathbb{R}}^n$
is subanalytic.
\end{lemma}

Let $\psi\colon X\to {\mathbb{R}}$ be a map.  By the  {\it support of $\psi$}, denoted by
${\rm supp}(\psi)$,  we mean the
closure of the set $\{ x\in X\mid \psi(x)\not= 0\}$.

\begin{lemma}
\label{orbsumma}
Let $X$ be a real analytic orbifold and let $\psi_i\colon X\to {\mathbb{R}}$,
$i\in {\mathbb{N}}$, be subanalytic maps such that $\{ {\rm supp}(\psi_i)\}_{
i\in {\mathbb{N}}}$ is locally finite. Then the map
$$
\psi\colon X\to {\mathbb{R}},\,\,\, x\mapsto \sum_{i=1}^\infty \psi_i(x),
$$
is subanalytic.
\end{lemma}

\section{Subanalytic partitions of unity for orbifolds}
\label{maps}

\noindent In this section we prove the existence of subanalytic
partitions of unity in the orbifold case.

\begin{lemma}
\label{forpartunitex}
Let $(\tilde{V},G,\varphi)$ be a chart of a real analytic orbifold $X$. 
Let $A$ and $B$ be disjoint closed subsets of $V=\varphi({\tilde{V}})$. Then there
is a subanalytic map $f\colon V\to {\mathbb{R}}$ such that $f\vert A=0$,
$f\vert B=1$ and $f(V)\subset [0,1]$.
\end{lemma}

\begin{proof} It is well-known (see e.g. Proposition 5.4 in \cite{Ka}) that
there is a $G$-invariant subanalytic map $\tilde{f}\colon \tilde{V}\to
{\mathbb{R}}$ such that $\tilde{f}\vert \varphi^{-1}(A)=0$,
$\tilde{f}\vert \varphi^{-1}(B)=1$ and $\tilde{f}(\tilde{V})\subset [0,1]$.
This map induces a subanalytic map $f\colon V\to{\mathbb{R}}$ such that
$f\circ\varphi=\tilde{f}$. Clearly, $f$ has the desired properties.
\end{proof}

\begin{definition}
\label{partunit}
Let $X$ be a real analytic orbifold. A {\it  subanalytic 
partition of unity} is a collection $\{ \lambda_i\}$ of subanalytic 
maps $X\to {\mathbb{R}}$ with the following properties:
\begin{enumerate}
\item $\lambda_i(x)\geq 0$ for every $x\in X$,
\item $\{ {\rm supp}(\lambda_i)\}$ is a locally finite cover of $X$, and
\item $\sum_i \lambda_i(x)=1$ for every $x\in X$.
\end{enumerate}
\end{definition}

A partition $\{ \lambda_i\}$ of unity is said to be {\it subordinate} to an open cover
$\{ U_j\}$ of $X$, if for every $\lambda_i$ there is a $U_j$ such that ${\rm supp}(\lambda_i)
\subset U_j$.

\begin{theorem}
\label{partunitex}
Let $X$ be a real analytic orbifold  and let ${\mathcal U}$ be an open cover of $X$.
Then $X$ has a subanalytic  partition of unity
subordinate to ${\mathcal U}$.
\end{theorem}

\begin{proof} Since $X$ is paracompact, ${\mathcal U}$ has a locally finite refinement by
basic open sets. Thus we may assume that ${\mathcal U}$ consists of basic open
sets and it suffices to find a partition of unity subordinate to such ${\mathcal U}$.  
Since $X$ is paracompact, ${\mathcal U}=\{ U_j\}$ has locally finite refinements
$\{ W_j\}$ and $\{ V_j\}$ by open sets $W_j$ and $V_j$, respectively, such that
$\overline{V_j}\subset W_j$ and $\overline{W}_j\subset U_j$ for every $j$.
According to Lemma \ref{forpartunitex}, there exists for every $j$ a subanalytic
map $f_j\colon X\to{\mathbb{R}}$ such that $f_j(X)\subset [0,1]$, $f_j$ is
identically one on $\overline{V_j}$ and ${\rm supp}(f_j)\subset {W}_j$. Since
$\{ V_j\}$ is a cover of $X$, it follows that the maps
$$
\lambda_i\colon X\to {\mathbb{R}},\,\,\, x\mapsto
{ {f_i(x)}\over { \Sigma_j f_j(x)} },
$$
are well defined. The conditions $(1)$, $(2)$ and $(3)$ of Definition
\ref{partunit} clearly hold. The maps $\lambda_i$ are subanalytic by Lemmas
\ref{orbsumma} and \ref{orbtulo}.
\end{proof}

Notice that  the maps $\lambda_i$
in Theorem \ref{partunitex} are constructed in such a way that
${\rm supp}(\lambda_i)\subset U_i$, for every $i$.

\section{Triangulation theorem for real analytic orbifolds}
\label{triangle}
\noindent We are now ready to prove the triangulation theorem for
real analytic orbifolds. The corresponding result for orbit spaces is
proved in Section 7 of \cite{Ka}.   We construct the triangulation by 
adapting the ideas in \cite{Ka} to the orbifold case.
To get started, we need a local result, which follows
from a theorem of Matumoto and Shiota:

\begin{theorem}
\label{matushi}
Let $G$ be a compact Lie group and let $M$ be a real analytic $G$-manifold.
Let $\pi\colon M\to M/G$ be the natural projection. Then there
exists a real analytic proper $G$-invariant map $f\colon M\to {\mathbb{R}}^n$,
where $n=2{\rm dim}(M)+1$, such that the induced map $\bar{f}\colon M/G\to f(M)$ is a homeomorphism. Moreover, if another subanalytic set structure on $M/G$ is given
by an inclusion $j\colon M/G\to {\mathbb{R}}^p$, $p\in{\mathbb{N}}$, such that
$j\circ \pi\colon M\to {\mathbb{R}}^p$ 
is a proper subanalytic map, then $j(M/G)$ and $f(M)$ are subanalytically homeomorphic.
\end{theorem}

\begin{proof}
\cite{MS2}, Theorem 3.1.
\end{proof} 

Recall that if $M$ and $N$ are real
analytic manifolds, if  $f\colon M\to N$ is a proper real analytic map
and if $A$ is a subanalytic subset of $M$, then the image $f(A)$ is a
subanalytic set in $N$ (Proposition 3.8 in \cite{Hi1}). Thus the map $f$ 
in Theorem \ref{matushi} really induces a subanalytic structure for $M/G$.

By a {\it subanalytic homeomorphism} we mean a  subanalytic
bijection whose inverse map is subanalytic. Assume $f\colon A\to B$ is a homeomorphism
between subanalytic subsets  $A$ and $B$ of two real analytic manifolds. 
If $f$ is subanalytic, then the inverse
map of $f$ is automatically subanalytic  and we call $A$ and $B$ {\it subanalytically
homeomorphic.} The orbifold case differs form the manifold case, since a
map between two orbifolds that is both subanalytic and a  homeomorphism
does not need to be a subanalytic homeomorphism. The reason is that the inverse
map does not need to be an orbifold map, i.e., it does not necessarily have the
local lifts as in Definition \ref{submap}.

\begin{lemma}
\label{homeom}
Let $X$ be a real analytic orbifold. Then there exists a subanalytic map 
$f_0\colon X\to{\mathbb{R}}^q$, for some $q\in{\mathbb{N}}$, such that
$f_0$ is a homeomorphism onto $f_0(X)$.
\end{lemma}

\begin{proof}   Since orbifolds are second countable, it follows that
the orbifold $X$ can be covered by  basic open sets $V_i$, 
$i\in{\mathbb{N}}$. By Theorem \ref{matushi}, for every orbifold chart 
$(\tilde{V}_i, G_i, \varphi_i)$, there is a real analytic proper $G_i$-invariant map
$\tilde{f}_i\colon {\tilde{V}}_i\to {\mathbb{R}}^s$, for some $s\in{\mathbb{N}}$, such that
the induced map $\bar{f}_i\colon V_i\to\tilde{ f}_i(\tilde{V}_i)$ is a homeomorphism. We can choose
$s=2{\rm dim}\ (X)+1$. 

By Theorem \ref{palais}, $\{ V_i\}_{i=1}^\infty$ has an open locally finite
refinement $\{ O_{j\beta} \mid \beta\in B_j, j=1,\ldots, k\}$, such that  
$O_{j\beta}\cap O_{j\beta'}=\emptyset$ if $\beta\not= \beta'$.  As pointed out
after Theorem \ref{palais}, we may assume that each $B_j\subset {\mathbb{N}}$.
For every $j$ we denote $\bigcup_{\beta\in B_j}O_{j\beta}$ by $O_j$.
For every $j$ we define a map $f_j\colon O_j\to {\mathbb{R}}^{s+1}$, by setting
$f_j(y)=(\bar{ f}_{i_0}(y),\beta)$ if $y\in O_{j\beta}$ and $i_0$ is the smallest $i$ for which
$O_{j\beta}\subset V_i$. Then $f_j$ is real analytic. Clearly, $f_j$ is an injection.

Since $X$ is paracompact, we can choose open covers 
$\{ W_{j\beta} \mid \beta\in B_j, j=1,\ldots, k\}$,
$\{ U_{j\beta} \mid \beta\in B_j, j=1,\ldots, k\}$ and
$\{ Y_{j\beta} \mid \beta\in B_j, j=1,\ldots, k\}$
of $X$ by open sets
$W_{j\beta}$,  $U_{j\beta}$ and $Y_{j\beta}$,  respectively, such that 
$\overline{U}_{j\beta}\subset {W}_{j\beta}$,
$\overline{W}_{j\beta}\subset {Y}_{j\beta}$ and
$\overline{Y}_{j\beta}\subset {O}_{j\beta}$ for every $j$ and $\beta$.
We write $W_j=\bigcup_{\beta\in B_j}W_{j\beta}$,
$U_j=\bigcup_{\beta\in B_j}U_{j\beta}$ and
$Y_j=\bigcup_{\beta\in B_j}Y_{j\beta}$ for every $j$.
Let $h_j\colon X\to [0,1]$ be a subanalytic map which is identically one
on $\overline{U}_j$ and vanishes outside $W_j$. Similarly, let
$h'_j\colon X\to [0,1]$ be a subanalytic map which is identically one
on $\overline{W}_j$ and vanishes outside $Y_j$.  Clearly, the maps
$$
f_{0j}\colon X\to {\mathbb{R}}^{s+1}, \,\,\,
f_{0j}(y)= 
\left\{\begin{array}{rl}
h'_j(y)f_j(y), & \,\,{\rm if}\,\,
y\in O_j     \\
0, & \,\,{\rm if}\,\, y\in X\setminus{O_j}
\end{array}\right.
$$
are subanalytic. Let
$$
f_0\colon X\to {\mathbb{R}}^p,\,\,\, y\mapsto
( h_1(y),\ldots, h_k(y), f_{01}(y),\ldots, f_{0k}(y)),
$$
where $p=k(s+2)$. Then $f_0$ is subanalytic by Proposition 8.7 in
\cite{Ka2}.

Since the maps $\bar{f}_i$ are embeddings, it follows
that the restriction $f_{0j}\vert W_{j\beta}$ is an embedding for every $\beta$.
Let $(x_d)_{d=1}^\infty$ be a sequence in $X$ such that $f_0(x_d)\to f_0(x)$,
for some $x\in X$.
Then $x\in U_{j\beta}$ for some $j$ and $\beta$ and $h_j(x)=1$. Since
$h_j(x_d)\to h_j(x)$, it follows that $h_j(x_d)>0$ for sufficiently large $d$.
Therefore $x_d\in W_{j\beta}$ for sufficiently large $d$. Since 
$f_{0j}\vert W_{j\beta}$ is an embedding and $f_{0j}(x_d)\to f_{0j}(x)$, 
it follows that $x_d\to x$. Consequently, $f_0$ is injective and the inverse
map $f_0^{-1}\colon f_0(X)\to X$ is continuous.
\end{proof}

\begin{theorem}
\label{firstthm}
Let $X$ be a real analytic  orbifold. Then there exists a proper
subanalytic map $f\colon X\to {\mathbb{R}}^n$ such that the induced
map $X\to f(X)$ is a  homeomorphism.  Thus  $f(X)$ is  a closed  subanalytic
subset of ${\mathbb{R}}^n$.
If $g\colon X\to {\mathbb{R}}^p$ is any proper subanalytic  map that also is
a topological embedding, 
then $g\circ f^{-1}\colon f(X)\to
g(X)$ is a subanalytic homeomorphism.
\end{theorem}

\begin{proof}
Let $\{ V_i\}_{i=1}^\infty$ be a cover of $X$ by basic open sets. We  may assume that
the closure of each $V_i$ is compact. By Theorem \ref{partunitex}, there is a subanalytic
partition of unity  $\{\lambda_i\}$ subordinate to $\{ V_i\}_{i=1}^\infty$. Let
$$
\lambda\colon X\to {\mathbb{R}},\,\,\, x\mapsto \sum_{i=1}^\infty 2^{-i}\lambda_i(x).
$$
By Lemma \ref{orbsumma}, $\lambda$ is subanalytic. Clearly, $0<\lambda(x)<1$,
for every $x\in X$. By Lemma \ref{homeom}, there is a subanalytic map
$f_0\colon X\to {\mathbb{R}}^q$, for some $q\in{\mathbb{N}}$, such that
$f_0$ is a homeomorphism onto the image $f_0(X)$. Let
$$
f_1\colon X\to {\mathbb{R}}^{q+1},\,\,\, x\mapsto (f_0(x), 1),
$$
and let
$$
f\colon X\to {\mathbb{R}}^{q+1},\,\,\, x\mapsto { {f_1(x)}\over {\lambda(x)}}.
$$
Then $f$ is subanalytic. Since $f_0$ is an embedding, it follows  that
also $f$ is an embedding.

We show that $f(X)$ is closed in ${\mathbb{R}}^{q+1}$. If the contrary is true,
there is a $y\in \overline{ f(X)}\setminus f(X)$. Let $(x_j)_{j=1}^\infty$ be a
sequence in $X$ such that $f(x_j)\to y$. Then  $({ {1}\over
{\lambda(x_j)}})_{j=1}^\infty$ converges to
the last coordinate  $y_{q+1}$ of $y$ in
${\mathbb{R}}^{q+1}$. Since $0<\lambda(x)<1$ for every $x\in X$, it follows that
$y_{q+1}>0$.  Then $f_1(x_j)\mapsto { {y}\over {y_{q+1}}}$. Assume first
${ {y}\over {y_{q+1}}}\notin f_1(X)$ and let $\{ U_m\}_{m=1}^\infty$ be a
neighborhood basis of ${ {y}\over {y_{q+1}}}$. Since the maps
$\lambda_i$ have compact supports, the sets $U'_m= U_m\setminus \bigcup_{i=1}^m
f_1({\rm supp}(\lambda_i))$ also form a neighborhood basis  of
${ {y}\over {y_{q+1}}}$. For every $m$ there exists an $x_{j_m}\in
\{ x_j\}_{j=1}^\infty$ such that $f_1(x_{j_m})\in U'_m$. We may choose 
$j_{m+1}>j_m$ for every $m$. Then $x_{j_m}\notin \bigcup_{i=1}^m
{\rm supp}(\lambda_i)$.  Thus  $\lambda(x_{j_m})\to 0$,
which is impossible. It follows that  ${ {y}\over {y_{q+1}}}\in f_1(X)$. Thus
${ {y}\over {y_{q+1}}}=f_1(x)$ for some $x\in X$. Consequently, 
$f_0(x_j)\mapsto f_0(x)$. Since $f_0$ is an embedding, it follows that
$x_j\mapsto x$. Thus $\lambda(x_j)\mapsto \lambda(x)$ and
$\lambda(x)={ {1}\over {y_{q+1}} }$. Therefore, $y={ {f_1(x)}\over {\lambda(x)}}
=f(x)$, and it follows that $f(X)$ is closed. Thus $f$ is 
a proper subanalytic map and  it follows
from Theorem 9.3 in \cite{Ka2} that $f(X)$ is a closed  subanalytic subset of 
${\mathbb{R}}^{q+1}$.

It remains  to show that the subanalytic structure on $X$
is unique. Therefore, let $g\colon X\to {\mathbb{R}}^p$
be a proper subanalytic map that also is a topological embedding, 
where $p\in{\mathbb{N}}$. Then
$g\circ f^{-1}\colon f(X)\to g(X)$ is a homeomorphism.   Clearly, the map $(f,g)\colon X\to 
{\mathbb{R}}^{q+1}\times {\mathbb{R}}^p$ is subanalytic and proper.
Theorem 9.3 in \cite{Ka2} indicates that  the graph ${\rm Gr}(g\circ f^{-1})=(f,g)(X)$ 
is a subanalytic subset of ${\mathbb{R}}^{q+1}\times {\mathbb{R}}^p$.
Thus $g\circ f^{-1}$ is a subanalytic map.
\end{proof}

Let $K$ be a simplicial complex. We denote by $\vert K\vert$ the
space of $K$. The corresponding open simplex of any simplex
$\sigma\in K$ is denoted by ${\rm int}\ \sigma$. The space of
any countable locally finite simplicial complex $K$ admits a
linear embedding as a closed subset of some euclidean space
${\mathbb{R}}^n$, see Theorem 3.2.9 in \cite{Sp}. If $e\colon
\vert K\vert\to {\mathbb{R}}^n$ is a closed  linear embedding, then the
image $e(\vert K\vert)$ is a subanalytic subset of ${\mathbb{R}}^n$.
If $e_\ast\colon \vert K\vert\to{\mathbb{R}}^m$ is another closed
linear embedding, then $e_\ast\circ e^{-1}\colon  e(\vert K\vert)\to
e_\ast(\vert K\vert)$ is a subanalytic homeomorphism. Hence the
closed linear embeddings induce a unique subanalytic structure on
$\vert K\vert$.

\begin{definition}
\label{tri}
Let $X$ be a real analytic orbifold. A {\it subanalytic triangulation} 
of $X$ is a pair of a  simplicial complex
$K$ and a homeomorphism $\tau\colon \vert K\vert\to X$ such that the
inverse map $\tau^{-1}\colon X\to \vert K\vert$ is subanalytic.
We say that the triangulation is {\it compatible} with  the  family $\{ X_i\}$ 
of subsets of $X$, if each
$X_i$ is a union of some $\tau({\rm int}\ \sigma)$, where $\sigma\in K$.
\end{definition}

\begin{theorem}
\label{hironthm}
Let $\{ X_i\}$ be a locally finite family of subanalytic subsets in
${\mathbb{R}}^n$ which are contained in a subanalytic closed set
$X$ in ${\mathbb{R}}^n$. Then there exists a locally finite simplicial complex $K$
and a subanalytic homeomorphism $\tau\colon \vert K\vert\to X$ such that
each $X_i$ is a union of some $\tau({\rm int}\ \sigma)$, where 
$\sigma\in K$.
\end{theorem}

\begin{proof} \cite{Hi}, Theorem on p. 180.
\end{proof}

\begin{theorem}
\label{thetheorem}
Let $X$ be a real analytic  orbifold. Then $X$ has a
subanalytic triangulation compatible with the strata of $X$.
\end{theorem}

\begin{proof}
Let $f\colon X\to {\mathbb{R}}^n$ be as in Theorem \ref{firstthm}.
The strata $X_{(H)}$ are subanalytic in $X$, by Lemma \ref{substrata}.
Clearly, they form a locally finite family in $X$.  Since $f$ is proper and 
subanalytic, it follows that the sets $f(X_{(H)})$ are subanalytic in
${\mathbb{R}}^n$ and that they form a locally finite family.
By Theorem \ref{hironthm},
there exists a simplicial complex $K$ and a subanalytic homeomorphism
$\tau\colon \vert K\vert\to f(X)$ such that each $f(X_{(H)})$ is a union of some
$\tau({\rm int}\ \sigma)$, where $\sigma\in K$. Thus $f^{-1}\circ\tau\colon
\vert K\vert \to X$ is a homeomorphism and each $X_{(H)}$ is a union of
some $f^{-1}\circ\tau({\rm int}\ \sigma)$, where $\sigma\in K$. The inverse
map $\tau^{-1}\circ f$ of $f^{-1}\circ \tau$ is subanalytic since $\tau^{-1}$
and $f$ are subanalytic and $\tau^{-1}$ is proper, see Corollary 9.4 in
\cite{Ka2}. 
\end{proof}

The subanalytic triangulation of $X$ is unique in the sense that
if $\tau_1\colon \vert K_1\vert\to X$ and $\tau_2\colon \vert K_2\vert\to X$
are two subanalytic triangulations of $X$, then $\tau_1^{-1}\circ
\tau_2\colon \vert K_2\vert\to \vert K_1\vert$ is a subanalytic homeomorphism
by Theorem \ref{firstthm}.

\section{Compatible differential structures}
\label{comp}

\noindent By a {\it ${\rm C}^r$-differential structure} on 
an orbifold $X$ we mean a maximal
 ${\rm C}^r$-atlas $\alpha$ on $X$. A ${\rm C}^s$-differential structure 
 $\beta$ on $X$, $s>r$, is called {\it compatible} with $\alpha$, if
 $\beta \subset\alpha$. In this case, every chart on $\beta$ is a chart on
 $\alpha$. Equivalently, it means that the identity map of $X$ is a
 ${\rm C}^r$-orbifold diffeomorphism $X(\alpha)\to X(\beta)$.
 
Let $M$ be a ${\rm C}^k$-manifold, $1\leq k\leq\omega$. If a Lie
group $G$ acts on $M$ via a ${\rm C}^k$-action, we call $M$ a
${\rm C}^k$-$G$-manifold.
The following results are needed to prove Theorem \ref{comba1}:

\begin{theorem}
\label{apucomb0}
Let $G$ be a finite group and let $M$ and $N$ be
${\rm C}^k$-$G$-manifolds, $2\leq k\leq\omega$. Then any 
${\rm C}^r$-differentiable $G$-equivariant
map $M\to N$, $1\leq r<k$,  can be approximated arbitrarily well in the
strong ${\rm C}^r$-topology by a ${\rm C}^k$-differentiable $G$-equivariant map.
\end{theorem}

\begin{proof} A special case of Theorem 1.2 in \cite{MS2}.
\end{proof}

\begin{theorem}
\label{apucomb1}
Let $G$ be a compact Lie group and let $M$ be a ${\rm C}^r$-$G$-manifold, 
$1\leq r\leq\infty$. Then, there is a ${\rm C}^k$-$G$-manifold 
$\tilde{M}$ which is ${\rm C}^r$ $G$-equivariantly diffeomorphic to $M$, 
$r<k\leq\omega$.
\end{theorem}

\begin{proof} Theorem 1.3 in \cite{MS2}.
\end{proof}

Let $1\leq r\leq\infty$, and let $M$ and $N$ be ${\rm C}^r$-$G$-manifolds,
where $G$ is a finite group.  We denote the set of ${\rm C}^r$-differentiable
$G$-equivariant maps $M\to N$ equipped with the strong, i.e., the Whitney
topology, by ${\rm C}^r_{G,{\rm S}}(M,N)$. For $r=\infty$, we denote the set of
${\rm C}^\infty$-differentiable $G$-equivariant maps $M\to N$ equipped with
the Cerf topology (a topology finer than the Whitney topology) by
${\rm C}^\infty_{G,{\rm C}}(M,N)$.  Then ${\rm C}^\omega(M,N)$ is dense in
${\rm C}^\infty_{\rm C}(M,N)$ and, consequently,
${\rm C}^\omega_G(M,N)$ is dense in ${\rm C}^\infty_{G,{\rm C}}(M,N)$,
for finite $G$.

\begin{lemma}
\label{apucomb2}
Let $U$ and $V$ be open sets in ${\mathbb{R}}^n$, for some
$n\in{\mathbb{N}}$. Assume a finite group $G$ acts ${\rm C}^r$-differentiably 
both on $U$ and on $V$, $1\leq r<\infty$.
Let $W$ be an open  $G$-invariant subset of $U$. Let $f\colon
U\to V$ be a ${\rm C}^r$-differentiable $G$-equivariant map with $V'=f(W)$ open.
Then the restriction $f\vert W$ has a neighborhood ${\mathcal N}$ in ${\rm C}^r_{G,{\rm S}}
(W,V')$ such that if $g_0\in {\mathcal N}$, then  the map
$$
T(g_0)=g\colon U\to V,
$$
where
$$
g(x)= g_0(x),\,\,\, {\rm if} \,\,\, x\in W\,\,\, {\rm and}\,\,\, g(x)=
f(x)\,\,\, {\rm if}\,\,\, x\in U\setminus W,
$$
is a ${\rm C}^r$-differentiable $G$-equivariant map, and
$T\colon {\mathcal N}\to {\rm C}^r_{G,{\rm S}}(U,V)$, $g_0\mapsto T(g_0)$, is continuous.
\end{lemma}

\begin{proof} 
It is clear that $T(g_0)$ is equivariant when $g_0$ is equivariant.
The rest of the claims follow immediately from Lemma 2.2.8 in \cite{Hir}.
\end{proof}

According to Theorem 2.2.9 in \cite{Hir}, every ${\rm C}^r$-manifold,
$1\leq r<\infty$, has a compatible ${\rm C}^s$-differential structure, where
$r<s\leq\infty$. We follow  Hirsch's proof to prove the corresponding
result for orbifolds:

\begin{theorem}
\label{comba1} 
Let $\alpha$ be a ${\rm C}^r$-differential structure on the
orbifold $X$, $r\geq 1$. For every $s$, $r<s\leq\infty$, there exists a compatible
${\rm C}^s$-differential structure $\beta\subset\alpha$ on $X$. 
\end{theorem}

\begin{proof}  Let ${\mathcal B}$ denote
the family of all pairs $(B,\beta)$, where $B$ is an open subset of
$X$ and $\beta$ is a ${\rm C}^s$-structure on $B$, compatible with
the ${\rm C}^r$-structure $B$ inherites from $X$. By Lemma 
\ref{apucomb1}, the basic open sets of $X$ have a compatible 
${\rm C}^s$-structure. Thus ${\mathcal B}$ is not empty.

Define 
$$
(B_1,\beta_1)\leq (B_2,\beta_2)
$$
if and only if:
{\begin{enumerate}
\item $B_1\subset B_2$.
\item The ${\rm C}^s$-structure $\beta_1$ on $B_1$ is the one induced from the
${\rm C}^s$-structure $\beta_2$ on $B_2$.
\end{enumerate}}
\noindent Then $\leq$ defines an order in ${\mathcal B}$. Let ${\mathcal C}$ be a chain in
${\mathcal B}$. We denote by ${\mathcal C}_1$ the family of all $B$ occurring as
the first coordinate of a pair in ${\mathcal C}$, and by ${\mathcal C}_2$ the family of
all $\beta$ occurring as the second coordinate of a pair in ${\mathcal C}$. Let
$$
B^\ast=\bigcup_{B\in {\mathcal C}_1}B,\,\,\, {\rm and}\,\,\, \beta'=\bigcup_{
\beta\in{\mathcal C}_2}\beta.
$$
Then $B^\ast$ is an open subset of $X$ and $\beta'$ is a
${\rm C}^s$-atlas on $B^\ast$, compatible with the ${\rm C}^r$-structure
on $B^\ast$. Let $\beta^\ast$ be the maximal ${\rm C}^s$-atlas on
$B^\ast$ generated by $\beta'$.  Then $(B^\ast, \beta^\ast)$ is an
upper bound for ${\mathcal C}$ in ${\mathcal B}$.  It now follows from
Zorn's lemma, that ${\mathcal B}$ has a maximal element $(B,\beta)$.
We show that $B=X$.

Assume that $B\not= X$. Then there is a chart $(\tilde{U}, G,\varphi)$
of $X$ such that $U\cap (X\setminus B)\not=\emptyset$, where
$U=\varphi(\tilde{U})$. We may assume that $U\cap B\not=\emptyset$,
since $U\cap B=\emptyset$ would contradict the maximality of $B$.
By Theorem \ref{apucomb1}, there is an open set
$\hat{U}$ of ${\mathbb{R}}^n$, where $n=\dim(X)$, on which $G$ acts
${\rm C}^s$-differentiably, and a $G$-equivariant
${\rm C}^r$-diffeomorphism $f\colon \hat{U}\to \tilde{U}$. 
Put $W'=U\cap B$,  $\tilde{W}'=\varphi^{-1}(W')$ and $W=
f^{-1}(\tilde{W}')$.

There now are two differential structures on $W'$: the ${\rm C}^r$-structure 
$\alpha$ and the compatible ${\rm C}^s$-structure $\beta
\subset \alpha$.  We shall find a $G$-equivariant ${\rm C}^r$-diffeomorphism 
$\theta\colon \hat{U}\to \tilde{U}$ such that
the restriction $\theta\vert W\colon W\to \tilde{W}'$ is a
${\rm C}^s$-diffeomorphism. Then $\theta$ induces a
${\rm C}^r$-diffeomorphism $\theta'\colon \hat{U}/G \to U$ such that
$\theta'\vert W/G\colon W/G\to W'$ is a ${\rm C}^s$-diffeomorphism. In that
case the chart $(\hat{U}, G, \theta'\circ\pi)$, where $\pi\colon \hat{U}\to \hat{U}/G$
is the natural projection, has ${\rm C}^s$ overlap with $\beta$.
The ${\rm C}^s$-atlas $\beta\cup (\hat{U}, G, \theta'\circ\pi)$ on $B\cup U$ is 
contained in $\alpha$, which contradicts the maximality of $(B,\beta)$.

To construct $\theta$, we use Lemma \ref{apucomb2} 
to obtain a
neighborhood ${\mathcal N}\subset {\rm C}^r_{G,{\rm S}}(W, \tilde{W}')$ of
$f\vert W\colon W\to \tilde{W}'$ with the following property: Whenever
$g_0\in {\mathcal N}$, the map $T(g_0)=g\colon \hat{U}\to \tilde{U}$ defined by
$$
g(x)= g_0(x),\,\,\, {\rm if} \,\,\, x\in W\,\,\, {\rm and}\,\,\, g(x)=
f(x)\,\,\, {\rm if}\,\,\, x\in \hat{U}\setminus W,
$$
is ${\rm C}^r$-differentiable and $G$-equivariant, and the resulting map
$$
T\colon {\mathcal N}\to {\rm C}^r_{G,\rm S}(\hat{U}, \tilde{U})
$$
is continuous. The set ${\rm Diff}^r_G(\hat{U},\tilde{U})$
of $G$-equivariant ${\rm C}^r$-diffeomorphisms is open in ${\rm C}^r_{G,{\rm S}}(
\hat{U},\tilde{U})$.
Since $T(f\vert W)$ is the diffeomorphism $f$, there is a neighborhood
${\mathcal N}_0\subset {\mathcal N}$ of $f\vert W$ such that $T({\mathcal N}_0)\subset
{\rm Diff}^r_G(\hat{U},\tilde{U})$. 
Now, by Theorem \ref{apucomb0}, there is a $G$-equivariant
${\rm C}^s$-diffeomorphism $\theta_0\in
{\mathcal N}_0$. The required map $\theta$ is then $T(\theta_0)$.
\end{proof}

Notice that for ${\rm C}^r$-differentiable quotient orbifolds the existence of a 
compatible ${\rm C}^s$-differential structure follows immediately from the
corresponding equivariant results.

\section{Triangulation theorem for differentiable orbifolds}
\label{real}

\noindent In this section we improve Theorem \ref{comba1} and
conclude that, in fact, every ${\rm C}^1$-orbifold has a compatible
real analytic structure.

\begin{lemma}
\label{realapu}
Let $G$ be a compact Lie group and let $M$ and $N$ be smooth
$G$-manifolds. Let $f\colon M\to N$ be a 
${\rm C}^\infty$-differentiable $G$-equivariant map and
let $U$ be a $G$-invariant open subset of $M$. Then there exists an open 
neighborhood ${\mathcal N}$ of $f\vert U$ in ${\rm C}^{\infty}_{G,{\rm C}}(U,N)$ such that
the following holds: If $h\in {\mathcal N}$ and $T(h)\colon M\to N$ is defined by
$$
T(h)(x)=h(x),\,\,\, {\rm if}\,\,\, x\in U\,\,\, {\rm and}\,\,\,
T(h)(x)=f(x),\,\,\, {\rm if}\,\,\, x\in M\setminus U,
$$
then $T(h)$ is a  ${\rm C}^\infty$-differentiable $G$-equivariant map. Furthermore,
$T\colon {\mathcal N}\to  {\rm C}^{\infty}_{G,{\rm C}}(M,N)$, $h\mapsto T(h)$,
is continuous.
\end{lemma}

\begin{proof} Lemma 8.1, in \cite{I}.
\end{proof}

\begin{theorem}
\label{combanal} 
Let $\alpha$ be a ${\rm C}^\infty$-differential structure on the
orbifold $X$. Then there exists a compatible
real analytic differential structure $\beta\subset\alpha$ on $X$. 
\end{theorem}

\begin{proof}
The proof is similar to the proof of  Theorem \ref{comba1}. 
The reference to Lemma \ref{apucomb2}  should be replaced by a reference to
Lemma \ref{realapu}.
\end{proof}

Theorems \ref{comba1} and \ref{combanal} imply:

\begin{cor}
\label{allcomb}
Let $\alpha$ be a ${\rm C}^1$-differential structure on the orbifold
$X$. Then there exists a compatible real analytic structure
$\beta\subset\alpha$ on $X$.
\end{cor}

In other words:

\begin{theorem} 
\label{combanal2} Let $1\leq r\leq \infty$. Then
every ${\rm C}^r$-orbifold is ${\rm C}^r$-diffeomorphic 
to a ${\rm C}^\omega$-orbifold.
\end{theorem}

Let $X$ be a ${\rm C}^r$-orbifold, $1\leq r\leq\infty$. If there is a real
analytic orbifold $Y$ and a ${\rm C}^r$ orbifold diffeomorphism 
$f\colon Y\to X$, then a subanalytic triangulation of $Y$ induces
a triangulation of $X$. We call such a triangulation of $X$
"subanalytic".

\begin{theorem}
\label{smoothtri}
Let $X$ be a ${\rm C}^r$-orbifold, $1\leq r\leq\infty$. Then $X$ has
a "subanalytic" triangulation.
\end{theorem}

\begin{proof}
The result follows immediately from Theorems \ref{combanal2}
and   \ref{thetheorem}.
\end{proof}


\begin{thebibliography}{999}
\bibitem{BM} Bierstone, E. and P.D. Milman,
        {\em Semianalytic and subanalytic sets,}
        Inst. Hautes \'Etudes Sci. Publ. Math. {\bf 67} (1988), 5--42.
\bibitem{Ha} Hardt, R.M.,
             {\em Triangulation of subanalytic sets and proper light subanalytic maps,}
             Invent. Math. {\bf 38} (1977), 207--217.
\bibitem{Hi1} Hironaka, H., 
               {\em Subanalytic sets - Number theory, algebraic geometry and
               commutative algebra, in honor of Y. Akizuki, Kinokuniya,} Tokyo,
               {\bf 1973}, 454--493.           
\bibitem{Hi} Hironaka, H.,
              {\em Triangulation of algebraic sets,}
              Proc. Sympos. Pure Math., Amer. Math. Soc. {\bf 29} (1975),
              165--185.              
\bibitem{Hir} Hirsch, M.W.,
        {\em Differential topology,}
         Springer-Verlag, New York--Berlin, 1976.
 \bibitem{I} Illman, S.,
       {\em The very strong ${\rm C}^\infty$ topology  on
 ${\rm C}^\infty(M,N)$ and $K$-equivariant maps,}
 Osaka J. Math. {\bf 40} (2003), 409--428.   
 \bibitem{Ka} Kankaanrinta, M.,
       {\em Proper real analytic actions of Lie groups on
      manifolds,}
       Ann. Acad. Sci. Fenn., Ser A I Math. Dissertationes {\bf
       83} (1991), 1--41.
\bibitem{Ka2} Kankaanrinta, M.,
          {\em  On subanalytic subsets of real analytic orbifolds,}
           http://arxiv.org/abs/1104.4653, submitted for publication.        
\bibitem{MS1} Matumoto, T. and M. Shiota,
               {\em Proper subanalytic transformation groups and unique triangulation of the
               orbit spaces,}
                Lecture Notes in Mathematics {\bf 1217}, Springer--Verlag,
                Berlin--Heidelberg, 1986, 290--302.     
\bibitem{MS2} Matumoto, T. and M. Shiota,
            {\em Unique triangulation of the orbit space of a differentiable  
             transformation group and its applications,} 
             Adv. Stud. Pure Math. {\bf 9} (1986), 41--55.
\bibitem{Pa} Palais, R.S., 
              {\em The classification of $G$-spaces,}
               Mem. Amer. Math. Soc. {\bf 36}, 1960.    
\bibitem{Pa2} Palais, R.S.,
               {\em On the existence of slices for actions of non-compact Lie groups,}
                Ann. of Math. (2) {\bf 73} (1961) 295--323.   
 \bibitem{Sp} Spanier, E.H.,
                {\em Algebraic topology,}
                Springer-Verlag, New York--Heidelberg--Berlin, 1966.                                    
       
\end{thebibliography}
\end {document}